\documentclass[a4paper]{article}

\usepackage{amsmath, amsthm, amsfonts, amssymb,accents}
\usepackage{graphicx,color}
\usepackage{srcltx}
\usepackage{eucal}

\theoremstyle{plain}
\newtheorem{theorem}{Theorem}[section]

\newtheorem{lemma}[theorem]{Lemma}

\theoremstyle{remark}

\numberwithin{equation}{section}

\newcommand{\eps}{\varepsilon}
\newcommand{\sii}{L_2}
\newcommand{\ie}{\emph{i.e.}}
\newcommand{\eg}{\emph{e.g.}}
\newcommand{\cf}{\emph{cf}}

\def\Om{\Omega}

\def\e{\varepsilon}

\def\p{\partial}

\def\E{\mbox{\rm e}}
\def\a{\alpha}

\def\Odr{\mathcal{O}}
\def\H{W_2}

\def\Hinf{W_\infty}

\def\di{\,d}

\def\iu{\mathrm{i}}

\def\RR{\mathbb{R}}
\def\CC{\mathbb{C}}

\DeclareMathOperator{\RE}{Re}
\DeclareMathOperator{\IM}{Im}

\DeclareMathOperator{\Dom}{\mathfrak{D}}

\newcounter{assumption}
\begin{document}

\title{\textbf{The effective Hamiltonian for thin layers
with non-Hermitian Robin-type boundary conditions}}
\author{%
Denis Borisov\,$^a$ \ and \ David Krej\v ci\v r\'\i k\,$^{b,c,d}$%
}
\date{\footnotesize
\begin{center}
\begin{quote}
\begin{enumerate}
{\it
\item[$a)$]
Bashkir State Pedagogical University,
October St.~3a, 450000 Ufa,
Russian Federation; \texttt{borisovdi@yandex.ru}
\vspace{-0.7ex}
\item[$b)$]
Department of Theoretical Physics,
Nuclear Physics Institute ASCR,
25068 \v Re\v z, Czech Republic;
\texttt{krejcirik@ujf.cas.cz}
\vspace{-0.7ex}
\item[$c)$]
Basque Center for Applied Mathematics,
Bizkaia Technology Park, Building 500,
48160 Derio, Kingdom of Spain
\vspace{-0.7ex}
\item[$d)$]
IKERBASQUE, Basque Foundation for Science,
Alameda Urquijo, 36, 5, 48011 Bilbao, Kingdom of Spain
}
\end{enumerate}
\end{quote}
\end{center}
%
24 February 2011
}
\maketitle

\begin{abstract}
\noindent
The Laplacian in an unbounded tubular neighbourhood of a hyperplane
with non-Hermitian complex-symmetric Robin-type boundary conditions
is investigated in the limit when the width of the neighbourhood diminishes.
We show that the Laplacian converges in a norm resolvent sense
to a self-adjoint Schr\"odinger operator in the hyperplane
whose potential is expressed solely
in terms of the boundary coupling function.
As a consequence, we are able to explain some peculiar
spectral properties of the non-Hermitian Laplacian
by known results for Schr\"odinger operators.
\end{abstract}
%

\section{Introduction}
%
There has been a growing interest in spectral properties
of differential operators in shrinking tubular neighbourhoods
of submanifolds of Riemannian manifolds,
subject to various boundary conditions.
This is partly motivated by the enormous progress in semiconductor physics,
where it is reasonable to try to model a complicated quantum Hamiltonian
in a thin nanostructure by an effective operator
in a lower dimensional substrate.
But the problem is interesting from the purely mathematical
point of view as well, because one deals with a singular limit
and it is not always obvious how
the information about the geometry and boundary conditions
are transformed into coefficients of the effective Hamiltonian.

The interest has mainly focused on self-adjoint problems,
namely on the Laplacian in the tubular neighbourhoods
with uniform boundary conditions of
Dirichlet \cite{DE,BMT,Friedlander-Solomyak_2007,Lampart-Teufel-Wachsmuth}
or Neumann~\cite{Schatzman_1996}
or a combination of those~\cite{K5}.
For more references see the review article~\cite{Grieser-2008}.
The purpose of the present paper is to show that one may obtain
an interesting self-adjoint effective operator in the singular limit
even if the initial operator is not Hermitian
and the geometry is rather trivial.

We consider an operator~$H_\eps$ which acts as the Laplacian
in a $d$-dimen\-sion\-al layer:
\begin{equation}\label{op}
  H_\eps u = -\Delta u
  \qquad\mbox{in}\qquad
  \Om_\e := \RR^{d-1} \times (0,\eps)
  \,,
\end{equation}
where $d \geqslant2$ and~$\eps$ is a small positive parameter,
subjected to non-Hermitian boundary conditions on~$\partial\Omega_\eps$.
Instead of considering the general problem, we rather restrict to
a special case of separated Robin-type boundary conditions
\begin{equation}\label{bc}
  \frac{\partial u}{\partial x_d} + \iu \;\! \alpha(x') \;\! u = 0
  \qquad\mbox{on}\qquad
  \p\Om_\e
  \,,
\end{equation}
where $x=(x_1,\ldots,x_{d-1},x_d)\equiv(x',x_d)$
denotes a generic point in~$\Omega_\eps$
and $\alpha$~is a real-valued bounded function.
More precisely, we consider~$H_\eps$ as
the m-sectorial operator~$H_\eps$ on $\sii(\Omega_\eps)$
which acts as~\eqref{op} in the distributional sense
on the domain consisting of functions~$u$
from the Sobolev space $\H^2(\Omega_\eps)$
satisfying the boundary conditions~\eqref{bc}.
We postpone the formal definition to the following section.

The model~$H_\eps$ in $d=2$ was introduced in~\cite{BK}
by the present authors.
In that paper, we developed a perturbation theory to study spectral properties
of~$H_\eps$ with~$\eps$ fixed in the regime when~$\alpha$
represents a small and local perturbation of constant~$\alpha_0$
(see below for the discussion of some of the results).
Additional spectral properties of~$H_\eps$ were further studied
in~\cite{KT} by numerical methods.
The present paper can be viewed as an addendum
by keeping~$\alpha$ (and~$d$) arbitrary
but sending rather the layer width~$\eps$ to zero.
We believe that the present convergence results
provide a valuable insight into the spectral phenomena
observed in the two previous papers.

The particular feature of the choice~\eqref{bc}
is that the boundary conditions are $\mathcal{PT}$-symmetric
in the sense that~$H_\eps$ commutes
with the product operator~$\mathcal{PT}$.
Here~$\mathcal{P}$ denotes the parity (space) reversal operator
$(\mathcal{P}u)(x):=u(x',\eps-x_d)$
and~$\mathcal{T}$ stands for the complex conjugation
$(\mathcal{T}u)(x):=\overline{u(x)}$;
the latter can be understood as the time reversal operator
in the framework of quantum mechanics.
The relevance of non-Hermitian $\mathcal{PT}$-symmetric
models in physics has been discussed in many papers recently,
see the review articles \cite{Bender_2007,Ali-review}.
Non-Hermitian boundary conditions of the type~\eqref{bc}
were considered in \cite{Kaiser-Neidhardt-Rehberg_2003b}
to model open (dissipative) quantum systems.
The role of~\eqref{bc} with constant~$\alpha$
in the context of perfect-transmission scattering
in quantum mechanics is discussed in~\cite{HKS}.

Another feature of~\eqref{bc} is that the spectrum of~$H_\eps$
``does not explode'' as the layer shrinks,
meaning precisely that the resolvent operator $(H_\eps+1)^{-1}$
admits a non-trivial limit in $\sii(\Omega_\eps)$ as $\eps \to 0$.
As a matter of fact, it is the objective of the present paper
to show that~$H_\eps$ converges in a norm resolvent sense
to the $(d-1)$-dimensional operator
\begin{equation}\label{limit.operator}
  H_0 := -\Delta + \alpha^2
  \qquad\mbox{on}\qquad
  \sii(\RR^{d-1})
  \,,
\end{equation}
which is a self-adjoint operator on the domain $\H^2(\RR^{d-1})$.
Again, we postpone the precise statement of the convergence,
which has to take into account that the operators~$H_\eps$ and~$H_0$
act on different Hilbert spaces, till the following section
(\cf~Theorem~\ref{th1.2}).
However, let us comment on spectral consequences of the result
already now.

First of all, we observe that a significantly non-self-adjoint
operator~$H_\eps$ converges, in the norm resolvent sense,
to a self-adjoint Schr\"odinger operator~$H_0$.
The latter contains the information about the non-self-adjoint
boundary conditions of the former in a simple potential term.
It follows from general facts \cite[Sec.~IV.3.5]{Kato} that
discrete eigenvalues of~$H_\eps$ either converge to
discrete eigenvalues of~$H_0$ or go to complex infinity
or to the essential spectrum of~$H_0$ as $\eps\to\infty$.

In particular, assuming that~$H_\eps$ and~$H_0$
have the same essential spectrum (independent as a set of~$\eps$),
the spectrum of~$H_\eps$ must approach
the real axis (or go to complex infinity) in the limit as $\eps \to 0$.
Although numerical computations performed in~\cite{KT}
suggest that~$H_\eps$ might have complex spectra in general,
perturbation analysis developed in~\cite{BK}
for the $2$-dimensional case
shows that both the essential spectrum
and weakly coupled eigenvalues are real.
The present paper demonstrates that the spectrum is real
also as the layer becomes infinitesimally thin,
for every $d \geqslant2$.
We would like to stress that the $\mathcal{PT}$-symmetry itself
is not sufficient to ensure the reality of the spectrum
and that the proof that a non-self-adjoint operator
has a real spectrum is a difficult task.

It is also worth noticing that the limiting operator~$H_0$
provides quite precise information about the spectrum
of~$H_\eps$ in the weak coupling regime for $d=2$
and $|\alpha|<\pi/\eps$ ($\eps$ fixed).
Indeed, let us consider the following special profile
of the boundary function:
$$
  \alpha_c(x') := \alpha_0 + c \;\! \beta(x')
  \,,
$$
where $\alpha_0$ is a real constant,
$\beta$~is a real-valued function of compact support
and~$c$ is a real parameter
(the regime of weak coupling corresponds to small~$c$).
Note that the essential spectrum of both~$H_\eps$ and~$H_0$
coincides with the interval $[\alpha_0^2,\infty)$,
for~$\beta$ is compactly supported (\cf~\cite[Thm.~2.2]{BK}).
It is proved in~\cite{BK}
that if $\alpha_0 \int_{\RR}\beta(x') \, dx'$ is negative,
then~$H_\eps$ possesses exactly one discrete real eigenvalue~$\mu(c)$
converging to~$\alpha_0^2$ as $c \to 0 +$
and the asymptotic expansion
$$
  \mu(c) = \alpha_0^2
  - c^2 \;\! \alpha_0^2 \left(\int_{\RR}\beta(x') \, dx'\right)^2
  + \mathcal{O}(c^3)
  \qquad\mbox{as}\qquad
  c \to 0 +
$$
holds true.
As a converse result, it is proved in~\cite{BK}
that there is no such a weakly coupled eigenvalue
if the quantity $\alpha_0 \int_{\RR}\beta(x') \, dx'$ is positive.
These weak coupling properties, including the asymptotics above,
are well known for the Schr\"odinger operator~$H_0$
with the potential given by~$\alpha_c^2$, 
see \cite{Gadylshin_2002}.

At the same time, the form of the potential in $H_0$~explains
some of the peculiar characteristics of~$H_\eps$ even for large~$c$.
As an example, let us recall that a highly non-monotone dependence
of the eigenvalues of~$H_\eps$ on the coupling parameter~$c$
was observed in the numerical analysis of~\cite{KT}.
As the parameter increases, a real eigenvalue typically emerges
from the essential spectrum, reaches a minimum
and then comes back to the essential spectrum again.
This behaviour is now easy to understand from
the non-linear dependence of the potential~$\alpha_c^2$ on~$c$.

On the other hand, we cannot expect that~$H_0$ represents
a good approximation of~$H_\eps$ for the values of parameters
for which~$H_\eps$ is known to possess complex eigenvalues~\cite{KT}.
It would be then desirable to compute the next to leading term
in the asymptotic expansion of~$H_\eps$ as $\eps \to 0$.

This paper is organized as follows.
In Section~\ref{Sec.main} we give a precise definition
of the operators~$H_\eps$ and~$H_0$ and state the norm
resolvent convergence of the former to the latter as $\eps \to 0$
(Theorem~\ref{th1.2}).
The rest of the paper consists of Section~\ref{Sec.proof}
in which a proof of the convergence result is given.

\section{The main result}\label{Sec.main}
%
We start with giving a precise definition of the operators~$H_\eps$ and~$H_0$.

The limiting operator~\eqref{limit.operator} can be immediately introduced
as a bounded perturbation of the free Hamiltonian on $\sii(\RR^{d-1})$,
which is well known to be self-adjoint on the domain $\H^2(\RR^{d-1})$.
For later purposes, however, we equivalently understand~$H_0$
as the operator
associated on $\sii(\RR^{d-1})$ with the quadratic form
\begin{align*}
  h_0[v] &:= \int_{\RR^{d-1}} |\nabla' v(x')|^2 \, dx'
  + \int_{\RR^{d-1}} \alpha(x')^2 \, |v(x')| \, dx'
  ,
  \\
  v \in \Dom(h_0) &:= \H^1(\RR^{d-1})
  \,.
\end{align*}
Here and in the sequel we denote by $\nabla'$ the gradient
operator in~$\RR^{d-1}$,
while~$\nabla$ stands for the ``full'' gradient in~$\RR^d$.

In the same manner, we introduce~$H_\eps$ as the m-sectorial operator
associated on $\sii(\Omega_\eps)$ with the quadratic form
\begin{align*}
  h_\eps[u] &:= \int_{\Omega_\eps} |\nabla u(x)|^2 \, dx
  + \iu \int_{\RR^{d-1}} \alpha(x') \, |u(x',\eps)|^2 \, dx'
  - \iu \int_{\RR^{d-1}} \alpha(x') \, |u(x',0)|^2 \, dx'
  ,
  \\
  u \in \Dom(h_0) &:= \H^1(\Omega_\eps)
  \,.
\end{align*}
Here the boundary terms are understood in the sense of traces.
Note that~$H_\eps$ is not self-adjoint unless $\alpha=0$
(in this case $H_\eps$ coincides with the Neumann Laplacian
in the layer~$\Omega_\eps$).
The adjoint of~$H_\eps$ is determined by simply changing
$\alpha$ to $-\alpha$ (or $\iu$ to $-\iu$)
in the definition of~$h_\eps$.
Moreover, $H_\eps$~is $\mathcal{T}$-self-adjoint \cite[Sec.~III.5]{Edmunds-Evans}
(or complex-symmetric~\cite{Garcia-Putinar}),
\ie~$H_\e^*=\mathcal{T}H_\e\mathcal{T}$.

The form~$h_\eps$ is well defined
under the mere condition that~$\alpha$ is bounded.
However, if we strengthen the regularity to $\a \in \Hinf^1(\RR^{d-1})$,
it can be shown by standard procedures (\cf~\cite[Sec.~3]{BK})
that~$H_\eps$ coincides with the operator described in the introduction,
\ie, it acts as the (distributional) Laplacian~\eqref{op}
on the domain formed by the functions~$u$ from $\H^2(\Omega_\eps)$
satisfying the boundary conditions~\eqref{bc} in the sense of traces.

The operator $H_0$~is clearly non-negative.
An analogous property for~$H_\eps$ is contained
in the following result
\begin{equation}\label{spectrum}
  \sigma(H_\eps) \subset
  \left\{
  z\in\CC \ \big| \ \RE z\geqslant 0, \,
  |\IM z|\leqslant 2\,\|\a\|_\infty \sqrt{\RE z}
  \right\}
  .
\end{equation}
Here and in the sequel we denote by $\|\cdot\|_\infty$ the supremum norm.
\eqref{spectrum}~can be proved exactly in the same way as in \cite[Lem.~3.1]{BK}
for $d=2$ by estimating the numerical range of~$H_\eps$.
In particular, the open left half-plane of~$\CC$ belongs
to the resolvent set of both~$H_\eps$ and~$H_0$.

Another general spectral property of~$H_\eps$, common with~$H_0$,
is that its residual spectrum is empty.
This is a consequence of the $\mathcal{T}$-self-adjointness
property of~$H_\eps$ as pointed out in~\cite[Corol.~2.1]{BK}.

Since $H_\eps$ and $H_0$ act on different Hilbert spaces,
we need to explain how the convergence of the corresponding
resolvent operators is understood.
We decompose our Hilbert space
into an orthogonal sum
\begin{equation}\label{direct}
  \sii(\Omega_\eps) = \mathfrak{H}_\eps \oplus \mathfrak{H}_\eps^\bot
  \,,
\end{equation}
where the subspace~$\mathfrak{H}_\eps$ consists of functions
from $\sii(\Omega_\eps)$ of the form $x \mapsto \psi(x')$,
\ie\ independent of the ``transverse'' variable~$x_d$.
The corresponding projection is given by
\begin{equation}\label{1.6}
  (P_\e u)(x)
  := \frac{1}{\e}\int\limits_{0}^{\e} u(x)\di x_d
\end{equation}
and it can be viewed as a projection
onto a constant function in the transverse variable.
We also write $P_\eps^\bot:=I-P_\eps$.
Since the functions from~$\mathfrak{H}_\eps$
depend on the ``longitudinal'' variables~$x'$ only,
$\mathfrak{H}_\eps$~can be naturally identified with $\sii(\RR^{d-1})$.
Hence, with an abuse of notations,
we may identify any operator on $\sii(\RR^{d-1})$
as the operator acting on $\mathfrak{H}_\eps \subset \sii(\Omega_\eps)$,
and vice versa.

The norm and the inner product in $\sii(\Omega_\eps)$
will be denoted by $\|\cdot\|_\eps$ and $(\cdot,\cdot)_\eps$, respectively.
We keep the same notation $\|\cdot\|_\eps$
for the operator norm on $\sii(\Omega_\eps)$.
The norm and the inner product in $\sii(\RR^{d-1})$
will be denoted by $\|\cdot\|$ and $(\cdot,\cdot)$, respectively,
\ie\ without the subscript~$\eps$.
All the inner products are assumed to be linear in the first component.
Finally, we denote the norm in $\H^1(\Om_\e)$ by $\|\cdot\|_{\e,1}$
and we keep the same notation for the norm of
bounded operators from $L_2(\Om_\e)$ to $\H^1(\Om_\e)$.

Now we are in a position to formulate the main result of this paper.
\begin{theorem}\label{th1.2}
Assume $\alpha \in \Hinf^1(\RR^{d-1})$.
Then the inequalities
\begin{align}
\big\|(H_\eps+1)^{-1}-(H_0+1)^{-1}P_\e\big\|_\eps
&\leqslant C \, \e,
\label{1.7}
\\
\big\|(H_\eps+1)^{-1}-(1+Q)(H_0+1)^{-1}P_\e\big\|_{\eps,1}
&\leqslant C(\eps) \, \e
\label{1.8}
\end{align}
hold true, where $Q(x):=-\iu \;\! \a(x') \;\! x_d$ and
\begin{align*}
C&:=\sqrt{\frac{1}{\pi^2}
  + \frac{(\|\nabla'\alpha\|_\infty +2\|\alpha\|_\infty)^2}{3}}
  \,,
  \\
C(\e)&:=\sqrt{
\frac{1}{\pi^2} + \left(
\frac{\|\nabla'\alpha\|_\infty+\|\alpha\|_\infty}{\sqrt{3}}
+ C_1(\e)
\right)^2
}
  \,,
  \\
  C_1(\e) &:= \sqrt{
\left(\frac{\eps \|\alpha\|_\infty^2}{2\sqrt{5}}\right)^2
+ \left(\frac{\eps \|\alpha\|_\infty^2}{2\sqrt{5}}
+ \frac{\|\a\|_\infty \sqrt{\|\a\|_\infty^2+\|\nabla'\a\|_\infty^2 \, \eps^2}}{\sqrt{3}}\right)^2
}
  \,.
\end{align*}
\end{theorem}

Let us discuss the result of this theorem.
It says that the operator~$H_\eps$ converges to~$H_0$
in the norm resolvent sense.
Note that, contrary to what happens for instance in the case of
uniform Dirichlet boundary conditions,
here we can choose the spectral parameter fixed
(\eg~$-1\in\rho(H_\eps)\cap\rho(H_0)$ as in the theorem)
and still get a non-trivial result.

If we treat the convergence of the resolvents
in the topology of bounded operators in $L_2(\Om_\e)$,
the estimate~\eqref{1.7} says that the rate of the convergence
is of order~$\Odr(\e)$.
At the same time, if we consider the convergence
as for the operators acting from $L_2(\Om_\e)$ into $\H^1(\Om_\e)$,
to keep the same rate of the convergence,
one has to use the function~$Q$.
This functions is to be understood as a corrector
needed to have the convergence in a stronger norm.
Such situation is well-known and it often happens
for singularly perturbed problems,
especially in the homogenization theory,
see, \eg, \cite{Borisov-Bunoiu-Cardone_2011,Birman-Suslina_2007,Zhikov_2006}.

\section{Proof of Theorem~\ref{th1.2}}\label{Sec.proof}
%
Throughout this section we assume $\alpha \in \Hinf^1(\RR^{d-1})$.
With an abuse of notation, we denote by the same symbol~$\alpha$
both the function on $\RR^{d-1}$
and its natural extension $x\mapsto\alpha(x')$ to $\RR^d$.

We start with two auxiliary lemmata.
The first tells us that the subspace~$\mathfrak{H}_\eps^\bot$
is negligible for~$H_\eps$ in the limit as $\eps \to 0$.
\begin{lemma}\label{lm2.1}
For any $f \in \sii(\Omega_\eps)$, we have
\begin{equation}\label{my.2.3}
  \big\|(H_\eps+1)^{-1} P_\eps^\bot f\big\|_{\eps,1}
  \leqslant \frac{\e}{\pi} \, \|P_\eps^\bot f\|_\eps
  \,.
\end{equation}
\end{lemma}
\begin{proof}
For any fixed $f \in \sii(\Omega_\eps)$,
let us set
$
  u:=(H_\eps+1)^{-1}P_\eps^\bot f
  \in \Dom(H_\eps) \subset \H^1(\Om_\e)
$.
In other words, $u$~satisfies the resolvent equation
\begin{equation*}
  \forall v \in \H^1(\Om_\e) \,, \qquad
  h_\eps(u,v)+ (u,v)_\eps
  = (P_\eps^\bot f,v)_\eps
  \,,
\end{equation*}
where $h_\e(\cdot,\cdot)$ denotes the sesquilinear form
associated with the quadratic form $h_\e[\cdot]$.
Choosing~$u$ for the test function~$v$
and taking the real part of the obtained identity,
we get
\begin{equation}\label{my.2.5}
  \|u\|_{\eps,1}^2
  = \RE \big(P_\eps^\bot f,u\big)_\eps
  = \RE \big(P_\eps^\bot f,P_\eps^\bot u\big)_\eps
  \leqslant\big\|P_\eps^\bot f\|_\e
  \|P_\eps^\bot u\|_\eps
  \,.
\end{equation}
Employing the decomposition
$
  u = P_\eps u + P_\eps^\bot u
$,
the left hand side of~\eqref{my.2.5} can be estimated as follows
\begin{equation}\label{my.2.5.1st}
  \|u\|_{\eps,1}^2
  \geqslant\|\nabla u\|_\eps^2
  \geqslant\|\partial_d u\|_\eps^2
  = \|\partial_d P_\eps^\bot u\|_\eps^2
  \geqslant(\pi/\eps)^2 \;\! \|P_\eps^\bot u\|_\eps^2
  \,.
\end{equation}
Here the last inequality follows from the variational
characterization of the second eigenvalue of the Neumann
Laplacian on $\sii((0,\eps))$ and Fubini's theorem.
Combining~\eqref{my.2.5.1st} with~\eqref{my.2.5}, we obtain
$$
  \|P_\eps^\bot u\|_\eps
  \leqslant(\eps/\pi)^2 \;\! \|P_\eps^\bot f\|_\eps
  \,.
$$
Finally, applying the obtained inequality
to the right hand side of~\eqref{my.2.5},
we conclude with
$$
  \|u\|_{\eps,1}^2
  \leqslant(\eps/\pi)^2 \;\! \|P_\eps^\bot f\|_\eps^2
  \,.
$$
This is equivalent to the estimate~\eqref{my.2.3}.
\end{proof}

In the second lemma we collect some elementary estimates
we shall need later on.
\begin{lemma}\label{lm2.2}
We have
\begin{align}
  |\E^{-\iu\a x_d}-1|
  &\leqslant \|\a\|_\infty \, x_d
  \,,
  \label{2.10} \\
  |\E^{-\iu\a x_d}-1+\iu\a x_d|
  &\leqslant \frac{1}{2} \, \|\a\|_\infty^2 \, x_d^2
  \,,
  \label{2.11} \\
  |\nabla(\E^{-\iu\a x_d}-1+\iu\a x_d)|
  &\leqslant \|\a\|_\infty \, x_d \,
  \sqrt{\|\a\|_\infty^2+\|\nabla'\a\|_\infty^2 \, x_d^2}
  \,.
  \label{2.12}
\end{align}
\end{lemma}
\begin{proof}
The estimates~\eqref{2.10} and~\eqref{2.11} are elementary
and we leave the proofs to the reader.
The last estimate~\eqref{2.12} follows from~\eqref{2.10}
and the identity
\begin{equation*}
\nabla(\E^{-\iu\a x_d}-1+\iu\a x_d)=\iu(1-\E^{-\iu\a x_d})
\begin{pmatrix}
x_d\nabla'\a
\\
\a
\end{pmatrix}
\end{equation*}
taken into account.
\end{proof}

We continue with the proof of Theorem~\ref{th1.2}.
Let $f \in \sii(\Omega_\eps)$.
Accordingly to~\eqref{direct}, $f$~admits the decomposition
$
  f = P_\eps f + P_\eps^\bot f
$
and we have
\begin{equation}\label{f.direct}
  \|f\|_\eps^2
  =\|P_\eps f\|_\eps^2 +\|P_\eps^\bot f\|_\eps^2
  =\e\;\!\|P_\eps f\|^2 +\|P_\eps^\bot f\|_\eps^2
  \,.
\end{equation}
We define $u:=(H_\eps+1)^{-1}f$ and make the decomposition
\begin{equation}\label{u.decomposition}
  u=u_0+u_1
  \qquad\mbox{with}\qquad
  u_0 := (H_\eps+1)^{-1}P_\eps f
  \,, \quad
  u_1 := (H_\eps+1)^{-1}P_\eps^\bot f
  \,.
\end{equation}
In view of Lemma~\ref{lm2.1},
$u_1$~is negligible in the limit as $\eps \to 0$,
\begin{equation}\label{u1.negligible}
  \|u_1\|_{\eps,1} \leqslant \frac{\e}{\pi} \, \|P_\eps^\bot f\|_\eps
  \,.
\end{equation}
It remains to study the dependence of~$u_0$ on~$\e$.
We construct $u_0$ as follows
\begin{equation}\label{u0.decomposition}
  u_0(x) = \E^{-\iu \a(x')x_d} \;\! w_0(x') + w_1(x)
  \,,\qquad\mbox{where}\qquad
  w_0:=(H_0+1)^{-1} P_\eps f
\end{equation}
and~$w_1$ is a function defined by this decomposition.

First, we establish a rather elementary bound for~$w_0$.
\begin{lemma}\label{Lem.w0}
We have
$$
  \|w_0\|_{\eps,1} \leqslant\|P_\eps f\|_\eps
  \,.
$$
\end{lemma}
\begin{proof}
By definition, $w_0$~satisfies the resolvent equation
\begin{equation}\label{w0.identity}
  \forall v \in \H^1(\RR^{d-1}) \,, \qquad
  h_0(w_0,v)+ (w_0,v)
  = (P_\eps f,v)
  \,,
\end{equation}
where $h_0(\cdot,\cdot)$ denotes the sesquilinear form
associated with the quadratic form $h_0[\cdot]$.
Choosing~$w_0$ for the test function~$v$,
we get
\begin{equation}\label{w0.equation}
  \|\nabla' w_0\|^2 + \|\alpha w_0\|^2 + \|w_0\|^2
  = (P_\eps f, w_0)
  \leqslant\|P_\eps f\| \|w_0\|
  \,.
\end{equation}
In particular,
$$
  \|w_0\| \leqslant\|P_\eps f\|
  \,.
$$
Using this estimate in the right hand side of~\eqref{w0.equation},
we get
$$
  \|\nabla' w_0\|^2 + \|w_0\|^2
  \leqslant\|P_\eps f\|^2
  \,.
$$
Reintegrating this inequality over $(0,\eps)$,
we conclude with the desired bound in~$\Omega_\eps$.
\end{proof}

It is more difficult to get a bound for~$w_1$.
\begin{lemma}\label{Lem.w1}
We have
$$
  \|w_1\|_{\eps,1}
  \leqslant C_1 \, \eps \, \|P_\eps f\|_\eps
  \qquad\mbox{with}\qquad
  C_0:= \frac{\|\nabla'\alpha\|_\infty
  + \|\alpha\|_\infty}{\sqrt{3}}
  \,.
$$
\end{lemma}
\begin{proof}
By definition, $u_0$~satisfies the resolvent equation
\begin{equation*}
  \forall v \in \H^1(\Om_\e) \,, \qquad
  h_\eps(u_0,v) + (u_0,v)_\eps
  = (P_\eps f,v)_\eps
  \,.
\end{equation*}
Choosing~$w_1$ for the test function~$v$
and using the decomposition~\eqref{u0.decomposition},
we get
\begin{equation}\label{w1.identity}
  h_\eps[w_1] + \|w_1\|_\eps^2
  = (P_\eps f,w_1)_\eps - h_\eps(u_0-w_1,w_1)
  - (u_0-w_1,w_1)_\eps
  =: F_\eps
  \,.
\end{equation}
It is straightforward to check that
\begin{align*}
  h_\eps(u_0-w_1,w_1)
  &= \big(\nabla w_0,\nabla e^{\iu \alpha x_d} w_1\big)_\eps
  + \big(\alpha^2 w_0, e^{\iu \alpha x_d} w_1\big)_\eps
  - F_\eps'
  \\
  &= - \big(w_0, e^{\iu \alpha x_d} w_1\big)_\eps
  + \big(P_\eps f,e^{\iu \alpha x_d} w_1\big)_\eps
  - F_\eps'
\end{align*}
with
$$
  F_\eps'
  := \iu \big(x_d\;\!w_0\;\! \nabla'\alpha, e^{\iu \alpha x_d} \nabla'w_1\big)_\e
  -\iu \big(\nabla' w_0,x_d\;\!e^{\iu \alpha x_d} w_1\;\! \nabla'\alpha\big)_\e
  \,.
$$
Here the first equality follows by algebraic manipulations
using an integration by parts,
while the second is a consequence of~\eqref{w0.identity},
with  $x' \mapsto e^{\iu \alpha(x') x_d} w_1(x',x_d)$
being the test function, and Fubini's theorem.
At the same time,
$
  (u_0-w_1,w_1)_\eps
  = (w_0,e^{\iu \alpha x_d} w_1)_\eps
$.
Hence,
$$
  F_\eps = F_\eps' + \big(P_\eps f,w_1-e^{\iu \alpha x_d} w_1\big)_\eps
  \,.
$$

We proceed with estimating~$F_\eps$:
\begin{align*}
  |F_\eps|
  &\leqslant\frac{\eps^{3/2}}{\sqrt{3}}
  \Big(
  \|\nabla'\alpha\|_\infty \|w_0\| \|\nabla'w_1\|_\eps
  + \|\nabla'\alpha\|_\infty \|\nabla'w_0\| \|w_1\|_\eps
  + \|\alpha\|_\infty \|P_\eps f\| \|w_1\|_\eps
  \Big)
  \\
  &\leqslant\frac{\eps^{3/2}}{\sqrt{3}}
 \Big(
  \|\nabla'\alpha\|_\infty \sqrt{\|w_0\|^2+\|\nabla'w_0\|^2}
  \sqrt{ \|\nabla'w_1\|_\eps^2+\|w_1\|_\eps^2}
  + \|\alpha\|_\infty \|P_\eps f\| \|w_1\|_\eps
  \Big)
  \\
  &\leqslant\frac{\eps}{\sqrt{3}}
  \Big(
  \|\nabla'\alpha\|_\infty \|w_0\|_{\eps,1}
  + \|\alpha\|_\infty \|P_\eps f\|_\eps
  \Big) \|w_1\|_{\eps,1}
  \,.
\end{align*}
Here the first inequality follows by the Schwarz inequality,
an explicit value of the integral of~$x_d^2$
and obvious bounds such as~\eqref{2.10}.

Finally, taking the real part of~\eqref{w1.identity} and using the above
estimate of~$|F_\eps|$, we get
$$
  \|w_1\|_{\eps,1}
  \leqslant\frac{\eps}{\sqrt{3}}
  \Big( \|\nabla'\alpha\|_\infty \|w_0\|_{\eps,1}
  + \|\alpha\|_\infty \|P_\eps f\|_\eps
  \Big)
  \,.
$$
The desired bound then follows by estimating $\|w_0\|_{\eps,1}$
by means of Lemma~\ref{Lem.w0}.
\end{proof}

Now we are in a position to conclude the proof of Theorem~\ref{th1.2}
by simply comparing~$u$ with~$w_0$.
As for the convergence in the topology of $\sii(\Omega_\eps)$,
we write
\begin{align*}
  \|u-w_0\|_\eps
  &= \|u_1+w_1+(\E^{-\iu\a x_d}-1)w_0\|_\eps
  \\
  &\leqslant\|u_1\|_\eps + \|w_1\|_\eps
  + \|(\E^{-\iu\a x_d}-1)w_0\|_\eps
  \,.
\end{align*}
Here the last term can be estimated using Lemma~\ref{lm2.2}
as follows
$$
  \|(\E^{-\iu\a x_d}-1)w_0\|_\eps
  \leqslant\frac{\eps^{3/2}}{\sqrt{3}} \, \|\alpha\|_\infty \|w_0\|
  = \frac{\eps}{\sqrt{3}} \, \|\alpha\|_\infty \|w_0\|_\eps
  \,.
$$
Hence, using~\eqref{u1.negligible}, Lemma~\ref{Lem.w1} and Lemma~\ref{Lem.w0},
we get the bound
\begin{align*}
  \|u-w_0\|_\eps
 & \leqslant \|u_1\|_{\eps,1} + \|w_1\|_{\eps,1}
  + \frac{\eps}{\sqrt{3}} \, \|\alpha\|_\infty
 \|w_0\|_\eps
\\
&\leqslant \e \left(\frac{1}{\pi} \, \|P_\eps^\bot f\|_\eps
  + \frac{1}{\sqrt{3}} \big(
 \|\nabla'\alpha\|_\infty+2 \;\! \|\alpha\|_\infty\big) \, \|P_\eps f\|_\eps
\right)
\\
&\leqslant C \, \eps \, \|f\|_\eps
\end{align*}
Here the last estimate follows by the Schwarz inequality
recalling~\eqref{f.direct} and holds with the constant~$C$
as defined in Theorem~\ref{th1.2}.
This proves~\eqref{1.7}.

As for the bound~\eqref{1.8}, we have
\begin{align*}
  \|u-(1+Q)w_0\|_{\eps,1}
  &= \|u_1+w_1+(\E^{-\iu\a x_d}-1+\iu \a x_d)w_0\|_{\eps,1}
  \\
  &\leqslant\|u_1\|_{\eps,1} + \|w_1\|_{\eps,1}
  + \|(\E^{-\iu\a x_d}-1+\iu \a x_d)w_0\|_{\eps,1}
  \,.
\end{align*}
Here the last term can be estimated using Lemma~\ref{lm2.2} as follows.
Employing the individual estimates
\begin{align*}
  \|(\E^{-\iu\a x_d}-1+\iu \a x_d)w_0\|_\eps
  &\leqslant\frac{\eps^{5/2}}{2\sqrt{5}} \, \|\alpha\|_\infty^2 \|w_0\|
  = \frac{\eps^2}{2\sqrt{5}} \, \|\alpha\|_\infty^2 \|w_0\|_\eps
  \,,
  \\
  \|(\E^{-\iu\a x_d}-1+\iu \a x_d)\nabla w_0\|_\eps
  &\leqslant\frac{\eps^{5/2}}{2\sqrt{5}} \, \|\alpha\|_\infty^2 \|\nabla'w_0\|
  = \frac{\eps^{2}}{2\sqrt{5}} \, \|\alpha\|_\infty^2 \|\nabla'w_0\|_\eps
  \,,
  \\
  \|w_0\nabla(\E^{-\iu\a x_d}-1+\iu \a x_d)\|_\eps
  &\leqslant\frac{\eps^{3/2}}{\sqrt{3}} \, \|\a\|_\infty
  \sqrt{\|\a\|_\infty^2+\|\nabla'\a\|_\infty^2 \, \eps^2}\, \|w_0\|
  \\
  &= \frac{\eps}{\sqrt{3}} \, \|\a\|_\infty
  \sqrt{\|\a\|_\infty^2+\|\nabla'\a\|_\infty^2 \, \eps^2}\, \|w_0\|_\eps
  \,,
\end{align*}
and the Schwarz inequality,
we may write
$$
  \|(\E^{-\iu\a x_d}-1+\iu \a x_d)w_0\|_{\eps,1}
  \leqslant C_1(\eps) \, \eps \, \|w_0\|_{\eps,1}
$$
with the same constant~$C_1(\eps)$ as defined in Theorem~\ref{th1.2}.
Consequently, using~\eqref{u1.negligible},
Lemma~\ref{Lem.w1}, Lemma~\ref{Lem.w0}
and the Schwarz inequality employing~\eqref{f.direct},
we get the bound
\begin{equation*}
  \|u-(1+Q)w_0\|_{\eps,1}
  \leqslant C(\eps) \, \eps \, \|f\|_\eps
\end{equation*}
with
\begin{equation}\label{constant'}
  C(\eps) := \sqrt{
  \frac{1}{\pi^2} + \Big(C_0 + C_1(\eps)\Big)^2
  }
  \,.
\end{equation}
Note that $C(\eps)$ coincides
with the corresponding constant of Theorem~\ref{th1.2}.
This concludes the proof of Theorem~\ref{th1.2}.

\section*{Acknowledgments}
D.B.~acknowledges the support and hospitality of the Basque Center for Applied
Mathematics in Bilbao. D.B. was partially supported by RFBR, by the grants of
the President of Russia for young scientists-doctors of sciences (MD-453.2010.1)
and for Leading Scientific School (NSh-6249.2010.1), and by the Federal Task
Program ``Research and educational professional community of innovation Russia''
(contract 02.740.11.0612). D.K. was partially supported by the Czech Ministry of
Education, Youth and Sports within the project LC06002 and by the GACR grant
P203/11/0701.

%
%

\end{document}